\documentclass[12pt]{amsart}

\usepackage[T1]{fontenc}
\usepackage{lmodern}
\usepackage{cite}

\usepackage{amsmath,amssymb,amsthm,mathtools}
\usepackage{mathrsfs}
\usepackage{bm}

\usepackage{microtype}
\usepackage{enumitem}
\usepackage[a4paper,inner=0.8in,outer=0.8in,top=1.2in,bottom=1.4in]{geometry}

\usepackage[
    colorlinks=true,
    linkcolor=blue,
    citecolor=blue,
    urlcolor=blue
]{hyperref}

\usepackage[nameinlink,capitalize,noabbrev]{cleveref}

\allowdisplaybreaks

\numberwithin{equation}{section}

\theoremstyle{plain}

\newtheorem{Theorem}{Theorem}[section]

\theoremstyle{definition}

\newtheorem{Definition}[Theorem]{Definition}

\theoremstyle{remark}

\newtheorem{Remark}[Theorem]{Remark}

\crefname{Theorem}{Theorem}{Theorems}
\crefname{Proposition}{Proposition}{Propositions}
\crefname{Lemma}{Lemma}{Lemmas}
\crefname{Corollary}{Corollary}{Corollaries}
\crefname{Definition}{Definition}{Definitions}
\crefname{Example}{Example}{Examples}
\crefname{Remark}{Remark}{Remarks}

\numberwithin{equation}{section}

\DeclareMathOperator{\supp}{supp}
\newcommand{\loc}{\mathrm{loc}}

\title[The Zero-Mass Problem]{On the Zero-Mass Problem of Bandle, Levine, and Zhang}

\author[M. Jleli]{Mohamed Jleli}
\address{(M. Jleli) Department of Mathematics, College of Science, King Saud University, Riyadh 11451, Saudi Arabia}
\email{jleli@ksu.edu.sa}

\author[B. Samet]{Bessem Samet}
\address{(B. Samet) Department of Mathematics, College of Science, King Saud University, Riyadh 11451, Saudi Arabia}
\email{bsamet@ksu.edu.sa}

\date{}

\subjclass[2020]{35K58, 35B33, 35B44}

\keywords{Semilinear heat equation; zero-mass forcing; critical exponent; global nonexistence}

\begin{document}

\begin{abstract}
We study the critical behavior of the inhomogeneous semilinear heat
equation
\[
u_t-\Delta u=|u|^p+w(x)
\qquad\text{in }(0,\infty)\times\mathbb{R}^N,
\]
where \(N\geq3\), \(p>1\), and \(w\in L^1(\mathbb{R}^N)\) is nontrivial
and has zero total mass:
\[
\int_{\mathbb{R}^N}w(x)\,dx=0.
\]
This zero-mass case was posed as an open problem by Bandle, Levine,
and Zhang in 2000 and has remained unresolved since then. We provide
a complete answer to this problem by proving nonexistence of global
weak solutions for
\[
1<p\leq\frac{N}{N-2}.
\]
Combined with the known supercritical existence result for sufficiently
small data, this shows that the critical exponent separating the
nonexistence and existence regimes remains
\[
p_c=\frac{N}{N-2},
\]
the same as in the positive-mass case.
\end{abstract}

\maketitle

\section{Introduction}

In \cite{BLZ}, Bandle, Levine, and Zhang studied the inhomogeneous
semilinear heat equation
\begin{equation}\label{BLZ-problem}
\begin{cases}
u_t-\Delta u=|u|^p+w(x),
& (t,x)\in(0,\infty)\times\mathbb{R}^N,\\[1mm]
u(0,x)=u_0(x),
& x\in\mathbb{R}^N,
\end{cases}
\end{equation}
where \(p>1\).  They established the following result.

\begin{Theorem}[Bandle--Levine--Zhang]\label{T-BLZ}
Let \(N\geq3\).

\medskip

\noindent
{\rm (i)} If
\[
1<p<\frac{N}{N-2}
\qquad\text{and}\qquad
\int_{\mathbb{R}^N}w(x)\,dx>0,
\]
then problem~\eqref{BLZ-problem} admits no global solution.

\medskip

\noindent
{\rm (ii)} If
\[
p=\frac{N}{N-2}
\qquad\text{and}\qquad
\int_{\mathbb{R}^N}w(x)\,dx>0,
\]
then, under suitable additional assumptions on \(w\) and the solution,
problem~\eqref{BLZ-problem} admits no global solution.

\medskip

\noindent
{\rm (iii)} If
\[
p>\frac{N}{N-2},
\]
then global solutions exist for sufficiently small \(w\) and \(u_0\)
with suitable decay at infinity, regardless of whether or not
\[
\int_{\mathbb{R}^N}w(x)\,dx>0.
\]

\medskip

\noindent
{\rm (iv)} For every \(p>1\), there exist \(u_0\) and \(w\) with
\[
\int_{\mathbb{R}^N}w(x)\,dx<0
\]
for which problem~\eqref{BLZ-problem} admits a global solution.
\end{Theorem}

Theorem~\ref{T-BLZ} identifies
\[
p_c=\frac{N}{N-2}
\]
as the critical exponent governing the global behavior of
problem~\eqref{BLZ-problem}. A striking feature of this result is that
the presence of a time-independent forcing term with positive total mass
changes the critical threshold. Indeed, even an arbitrarily small
forcing term \(w\) satisfying
\[
\int_{\mathbb{R}^N}w(x)\,dx>0
\]
shifts the classical Fujita exponent~\cite{Fujita}
\[
p_F=1+\frac{2}{N}
\]
to
\[
p_c=\frac{N}{N-2}
=1+\frac{2}{N-2}.
\]
Thus, the total mass of the inhomogeneous term plays a decisive role in
the global behavior of solutions, independently of the magnitude of the
forcing term.

In Section~5 of the same work~\cite{BLZ}, Bandle, Levine, and Zhang
left open the case where the forcing term is nontrivial but has zero
total mass, namely,
\[
w\not\equiv0
\qquad\text{and}\qquad
\int_{\mathbb{R}^N}w(x)\,dx=0.
\]
More precisely, they asked whether this case exhibits the same behavior
as the positive-mass regime or rather that of the negative-mass regime.

Despite the numerous contributions devoted to the critical behavior of
inhomogeneous semilinear evolution equations, these studies typically
assume that the forcing term is either nonnegative and nontrivial or has
positive total mass; see, for instance,
\cite{BorikhanovTorebek,FHS,JSS,KumarTorebek,OzaSuragan,Pinsky,
TobakhanovTorebek,WYY}.
By contrast, the  zero-mass case has not been settled in full
generality.

Notice that Piccirillo, Toscano, and Toscano
\cite{PiccirilloToscanoToscano} stated that their results also address
the open problem of Bandle, Levine, and Zhang. They consider an
inhomogeneous term \(f=f(x,t)\), but their assumptions do not cover the
original zero-mass setting of \cite{BLZ}. Indeed, the nonexistence result
relevant to this problem requires, among other conditions,
\[
f^-\in L^1(\mathbb{R}^N\times[0,\infty))
\qquad\text{and}\qquad
u_0^-\in L^1(\mathbb{R}^N),
\]
together with a nontrivial lower bound on \(f^+\) for large \(|x|\).
For a time-independent forcing term \(w=w(x)\not\equiv0\) satisfying
\[
w\in L^1(\mathbb{R}^N)
\qquad\text{and}\qquad
\int_{\mathbb{R}^N}w(x)\,dx=0,
\]
the first condition is necessarily violated, since
\[
\int_0^\infty\int_{\mathbb{R}^N}w^-(x)\,dx\,dt=+\infty.
\]
Thus, the zero-mass problem posed in \cite{BLZ} remains open in the
generality formulated there.

In this paper, we provide a complete answer to this open problem.
Namely, we prove that every nontrivial integrable forcing term with zero
total mass prevents the existence of global weak solutions in the
subcritical and critical ranges. Combined with the supercritical
existence result of Bandle, Levine, and Zhang
(Theorem~\ref{T-BLZ}{\rm (iii)}), this shows that the zero-mass case
exhibits the same critical threshold as the positive-mass regime.

Our main result reads as follows.

\begin{Theorem}\label{T-main}
Let \(N\geq3\) and let \(w\in L^1(\mathbb{R}^N)\) satisfy
\[
w\not\equiv0
\qquad\text{and}\qquad
\int_{\mathbb{R}^N}w(x)\,dx=0.
\]
If
\[
1<p\leq\frac{N}{N-2},
\]
then the equation
\begin{equation}\label{M-eq}
u_t-\Delta u=|u|^p+w(x)
\qquad\text{in }(0,\infty)\times\mathbb{R}^N
\end{equation}
admits no global weak solution.
\end{Theorem}

\begin{Remark}
No initial condition is imposed in Theorem~\ref{T-main}. In particular,
the nonexistence conclusion is independent of the choice of the initial
data.
\end{Remark}

We next briefly describe the main idea behind the proof of
Theorem~\ref{T-main}. To motivate our approach, we explain why the
argument of Bandle, Levine, and Zhang~\cite{BLZ} does not directly
extend to the zero-mass case. Indeed, their
approach relies on a family of cut-off functions \(\phi_R\) satisfying
\[
\int_{\mathbb{R}^N} w(x)\phi_R(x)\,dx
\longrightarrow
\int_{\mathbb{R}^N} w(x)\,dx
\qquad\text{as }R\to\infty.
\]
When the total mass is positive, this yields a uniform positive lower
bound
\[
\int_{\mathbb{R}^N} w(x)\phi_R(x)\,dx\geq\delta>0
\]
for all sufficiently large \(R\), which is a key ingredient in their
nonexistence arguments. In the zero-mass case, however,
\[
\int_{\mathbb{R}^N} w(x)\phi_R(x)\,dx\longrightarrow0,
\]
and such a positive lower bound is no longer available.

To overcome this difficulty, we exploit the fact that \(w\not\equiv0\).
Hence there exists a function \(\psi\in C_c^\infty(\mathbb{R}^N)\)
such that
\[
\int_{\mathbb{R}^N}w(x)\psi(x)\,dx>0.
\]
We then perturb the constant function by setting
\[
H_\varepsilon(x)=1+\varepsilon\psi(x),
\]
with \(\varepsilon>0\) sufficiently small. Since
\[
\int_{\mathbb{R}^N}w(x)\,dx=0,
\]
we obtain
\[
\int_{\mathbb{R}^N}w(x)H_\varepsilon(x)\,dx
=
\varepsilon
\int_{\mathbb{R}^N}w(x)\psi(x)\,dx>0.
\]
Thus, the perturbation \(H_\varepsilon\) restores a positive weighted
mass, while introducing only the error
\[
\Delta H_\varepsilon=\varepsilon\Delta\psi.
\]
We then combine \(H_\varepsilon\) with suitable large-scale cut-off
functions. Since \(\psi\) is compactly supported, for sufficiently large
cut-off radii the cut-off is identically \(1\) on \(\supp\psi\), while
its transition region is disjoint from \(\supp\Delta\psi\). This allows
the contributions of \(\varepsilon\Delta\psi\) and of the derivatives
of the cut-off to be estimated separately.

The detailed proof of Theorem~\ref{T-main} is given in the next section. 

\section{Proof of Theorem~\ref{T-main}}\label{S-proof}

Throughout the rest of the paper, \(C>0\) denotes a generic constant,
independent of \(x\), \(t\), \(R\), \(T\), and \(\varepsilon\), whose
value may change from line to line. For nonnegative quantities \(a\)
and \(b\), we write \(a\preceq b\) if \(a\leq Cb\), and \(a\asymp b\)
if \(a\preceq b\) and \(b\preceq a\). For \(p>1\), we denote by
\(p'=\frac{p}{p-1}\) the conjugate exponent of \(p\). For \(r>0\),
\(B(0,r)\) denotes the open ball in \(\mathbb{R}^N\) centered at the
origin with radius \(r\).

We first define the notion of a global weak solution that will be used
throughout the paper. 

\begin{Definition}\label{D-weak}
Let \(p>1\) and \(w\in L^1_{\loc}(\mathbb{R}^N)\). A function
\[
u\in L^p_{\loc}\bigl((0,\infty)\times\mathbb{R}^N\bigr)
\]
is called a global weak solution of \eqref{M-eq} if, for every
\(\varphi\in C_c^{1,2}\bigl((0,\infty)\times\mathbb{R}^N\bigr)\),
\(\varphi\geq0\), one has
\begin{equation}\label{ws}
\int_0^\infty\int_{\mathbb{R}^N}
|u|^p\varphi\,dx\,dt
+
\int_0^\infty\int_{\mathbb{R}^N}
w(x)\varphi\,dx\,dt
=
-\int_0^\infty\int_{\mathbb{R}^N}
u\varphi_t\,dx\,dt
-
\int_0^\infty\int_{\mathbb{R}^N}
u\Delta\varphi\,dx\,dt.
\end{equation}
\end{Definition}

\begin{Remark}
Since every
\(\varphi\in C_c^{1,2}\bigl((0,\infty)\times\mathbb{R}^N\bigr)\)
has support bounded away from \(t=0\), no initial term appears in
\eqref{ws}.
\end{Remark}

\begin{proof}[Proof of Theorem~\ref{T-main}]

Assume, by contradiction, that \eqref{M-eq} admits a global weak
solution \(u\).

Let
\(\varphi\in C_c^{1,2}\bigl((0,\infty)\times\mathbb{R}^N\bigr)\),
\(\varphi\geq0\), and define
\begin{align}
J_1(\varphi)
&=
\int_0^\infty\int_{\mathbb{R}^N}
\frac{|\varphi_t|^{p'}}{\varphi^{p'-1}}\,dx\,dt,
\label{J1}\\
J_2(\varphi)
&=
\int_0^\infty\int_{\mathbb{R}^N}
\frac{|\Delta\varphi|^{p'}}{\varphi^{p'-1}}\,dx\,dt.
\label{J2}
\end{align}
Assume that
\[
J_1(\varphi)+J_2(\varphi)<\infty.
\]
By \eqref{ws} and Young's inequality,
\[
\begin{aligned}
&\int_0^\infty\int_{\mathbb{R}^N}|u|^p\varphi\,dx\,dt
+
\int_0^\infty\int_{\mathbb{R}^N}w(x)\varphi\,dx\,dt\\
&\quad\leq
\int_0^\infty\int_{\mathbb{R}^N}|u||\varphi_t|\,dx\,dt
+
\int_0^\infty\int_{\mathbb{R}^N}|u||\Delta\varphi|\,dx\,dt\\
&\quad\leq
\int_0^\infty\int_{\mathbb{R}^N}|u|^p\varphi\,dx\,dt
+
C\bigl(J_1(\varphi)+J_2(\varphi)\bigr).
\end{aligned}
\]
Therefore,
\begin{equation}\label{apriori}
\int_0^\infty\int_{\mathbb{R}^N}w(x)\varphi\,dx\,dt
\preceq
J_1(\varphi)+J_2(\varphi).
\end{equation}

Since \(w\not\equiv0\), there exists
\(\psi\in C_c^\infty(\mathbb{R}^N)\) such that, replacing \(\psi\) by
\(-\psi\) if necessary,
\begin{equation}\label{def-A}
A=\int_{\mathbb{R}^N}w(x)\psi(x)\,dx>0.
\end{equation}
For \(\varepsilon>0\), to be chosen sufficiently small, define
\[
H_\varepsilon(x)=1+\varepsilon\psi(x).
\]
As \(\psi\) is bounded, we may choose \(\varepsilon>0\) so that
\begin{equation}\label{Heps-bounds}
\frac12\leq H_\varepsilon(x)\leq\frac32
\qquad\text{for all }x\in\mathbb{R}^N.
\end{equation}
Using \(\int_{\mathbb{R}^N}w(x)\,dx=0\) together with
\eqref{def-A}, we obtain
\begin{equation}\label{weighted-mass}
\int_{\mathbb{R}^N}w(x)H_\varepsilon(x)\,dx
=
\varepsilon A>0.
\end{equation}
Moreover,
\[
\Delta H_\varepsilon=\varepsilon\Delta\psi,
\]
and hence, by \eqref{Heps-bounds},
\begin{equation}\label{Heps-error}
\int_{\mathbb{R}^N}
\frac{|\Delta H_\varepsilon|^{p'}}
{H_\varepsilon^{p'-1}}\,dx
\preceq
\varepsilon^{p'}.
\end{equation}

Let \(\eta\in C_c^\infty(\mathbb{R})\) satisfy
\[
0\leq\eta\leq1,\qquad
\eta\not\equiv0,\qquad
\supp\eta\subset(1,2),
\]
and fix a positive integer \(k>2p'\). For \(T>0\), define
\[
\eta_T(t)=\eta^k\left(\frac{t}{T}\right).
\]
A change of variables gives
\begin{equation}\label{etaT-int}
\int_0^\infty\eta_T(t)\,dt\asymp T,
\end{equation}
and
\begin{equation}\label{etaT-der}
\int_0^\infty
\frac{|\eta_T'(t)|^{p'}}
{\eta_T(t)^{p'-1}}\,dt
\preceq T^{1-p'}.
\end{equation}

We first consider the subcritical case. Assume that
\[
1<p<\frac{N}{N-2}.
\]
Choose \(\xi\in C^\infty([0,\infty))\) such that
\[
0\leq\xi\leq1,\qquad
\xi=1\ \text{on }[0,1/2],\qquad
\xi=0\ \text{on }[1,\infty).
\]
For \(R>0\), set
\[
\theta_R(x)=\xi^k\left(\frac{|x|}{R}\right),
\qquad
\phi_R(x)=H_\varepsilon(x)\theta_R(x).
\]

Since \(\psi\) is compactly supported, there exists \(R_0>0\) such that
\begin{equation}\label{supp-psi-ball}
\supp\psi\subset B(0,R_0).
\end{equation}
For \(R>2R_0\), the properties of \(\xi\) imply that
\begin{equation}\label{1-supp-psi}
\theta_R=1
\qquad\text{on }\supp\psi.
\end{equation}
Hence,
\begin{equation}\label{phiR-decomp}
\phi_R
=
\theta_R+\varepsilon\psi,
\end{equation}
and consequently
\begin{equation}\label{Delta-phiR}
\Delta\phi_R
=
\Delta\theta_R+\varepsilon\Delta\psi.
\end{equation}
Moreover,
\[
\supp(\Delta\theta_R)
\subset
\overline{B(0,R)\setminus B(0,R/2)}.
\]
Therefore, \eqref{supp-psi-ball} and \(R>2R_0\) yield
\begin{equation}\label{disjoint-supports-psi}
\supp(\Delta\theta_R)\cap\supp\psi=\varnothing.
\end{equation}
Since \(\supp(\Delta\psi)\subset\supp\psi\), we also have
\begin{equation}\label{disjoint-supports}
\supp(\Delta\theta_R)\cap\supp(\Delta\psi)=\varnothing.
\end{equation}

From \eqref{Heps-bounds} and the properties of \(\theta_R\),
\begin{equation}\label{phiR-int}
\int_{\mathbb{R}^N}\phi_R(x)\,dx
\asymp R^N.
\end{equation}

On \(\supp(\Delta\theta_R)\), we have \(|x|\asymp R\). Hence, the
definition of \(\theta_R\) and the properties of \(\xi\) yield
\[
|\Delta\theta_R(x)|
\preceq
R^{-2}\xi^{k-2}\left(\frac{|x|}{R}\right).
\]
Consequently,
\[
\frac{|\Delta\theta_R|^{p'}}
{\theta_R^{p'-1}}
\preceq
R^{-2p'}
\xi^{k-2p'}\left(\frac{|x|}{R}\right).
\]
Using \(k>2p'\) and \(0\leq\xi\leq1\), we obtain
\begin{equation}\label{thetaR-est}
\int_{\mathbb{R}^N}
\frac{|\Delta\theta_R|^{p'}}
{\theta_R^{p'-1}}\,dx
\preceq
R^{N-2p'}.
\end{equation}

From \eqref{phiR-decomp}, \eqref{disjoint-supports-psi}, and
\eqref{1-supp-psi}, we have
\[
\phi_R=\theta_R
\quad\text{on }\supp(\Delta\theta_R),
\qquad
\phi_R=H_\varepsilon
\quad\text{on }\supp(\Delta\psi).
\]
By \eqref{Delta-phiR} and \eqref{disjoint-supports},
\(\Delta\theta_R\) and \(\varepsilon\Delta\psi\) have disjoint supports.
Hence,
\[
\int_{\mathbb{R}^N}
\frac{|\Delta\phi_R|^{p'}}
{\phi_R^{p'-1}}\,dx
=
\int_{\supp(\Delta\theta_R)}
\frac{|\Delta\phi_R|^{p'}}
{\phi_R^{p'-1}}\,dx
+
\int_{\supp(\Delta\psi)}
\frac{|\Delta\phi_R|^{p'}}
{\phi_R^{p'-1}}\,dx.
\]
Using \eqref{Delta-phiR}, \eqref{disjoint-supports}, and the preceding
identities, we obtain
\[
\int_{\mathbb{R}^N}
\frac{|\Delta\phi_R|^{p'}}
{\phi_R^{p'-1}}\,dx
=
\int_{\supp(\Delta\theta_R)}
\frac{|\Delta\theta_R|^{p'}}
{\theta_R^{p'-1}}\,dx
+
\int_{\supp(\Delta\psi)}
\frac{|\Delta H_\varepsilon|^{p'}}
{H_\varepsilon^{p'-1}}\,dx.
\]
Combining \eqref{Heps-error} and \eqref{thetaR-est}, we obtain
\begin{equation}\label{phiR-est}
\int_{\mathbb{R}^N}
\frac{|\Delta\phi_R|^{p'}}
{\phi_R^{p'-1}}\,dx
\preceq
R^{N-2p'}+\varepsilon^{p'}.
\end{equation}

Moreover, \eqref{def-A} and \eqref{phiR-decomp} give
\[
\int_{\mathbb{R}^N}w(x)\phi_R(x)\,dx
=
\int_{\mathbb{R}^N}w(x)\theta_R(x)\,dx+\varepsilon A.
\]
Since \(0\leq\theta_R\leq1\) and \(\theta_R(x)\to1\) pointwise, the
dominated convergence theorem yields
\[
\int_{\mathbb{R}^N}w(x)\theta_R(x)\,dx\to0
\qquad\text{as }R\to\infty.
\]
Therefore, for all sufficiently large \(R\),
\begin{equation}\label{w-phiR-lower}
\int_{\mathbb{R}^N}w(x)\phi_R(x)\,dx
\geq
\frac{\varepsilon A}{2}.
\end{equation}

For \(R,T\gg1\), define
\[
\varphi_{R,T}(t,x)=\eta_T(t)\phi_R(x).
\]
Then
\[
\varphi_{R,T}
\in C_c^{1,2}\bigl((0,\infty)\times\mathbb{R}^N\bigr),
\qquad
\varphi_{R,T}\geq0.
\]

Using Fubini's theorem together with \eqref{etaT-der},
\eqref{phiR-int}, \eqref{etaT-int}, and \eqref{phiR-est}, we obtain
\begin{equation}\label{J1-sub}
J_1(\varphi_{R,T})
\preceq
T^{1-p'}R^N
\end{equation}
and
\begin{equation}\label{J2-sub}
J_2(\varphi_{R,T})
\preceq
T\left(R^{N-2p'}+\varepsilon^{p'}\right).
\end{equation}
Moreover, \eqref{etaT-int} and \eqref{w-phiR-lower} yield
\begin{equation}\label{w-varphi-lower}
\varepsilon A T
\preceq
\int_0^\infty\int_{\mathbb{R}^N}
w(x)\varphi_{R,T}(t,x)\,dx\,dt.
\end{equation}

Applying \eqref{apriori} with \(\varphi=\varphi_{R,T}\), and using
\eqref{J1-sub}, \eqref{J2-sub}, and \eqref{w-varphi-lower}, we obtain
\[
\varepsilon A
\preceq
T^{-p'}R^N
+
R^{N-2p'}
+
\varepsilon^{p'}.
\]
Since \(p'>1\) and \(A>0\), choosing \(\varepsilon>0\) sufficiently
small allows the last term to be absorbed into the left-hand side.
Hence,
\[
\varepsilon A
\preceq
T^{-p'}R^N+R^{N-2p'}.
\]
Since \(1<p<\frac{N}{N-2}\), we have \(N-2p'<0\). Letting
\(R\to\infty\) yields a contradiction with \(\varepsilon A>0\).

We next consider the critical case. Assume that
\[
p=\frac{N}{N-2}.
\]
Choose \(\zeta\in C^\infty(\mathbb{R})\) such that
\[
0\leq\zeta\leq1,\qquad
\zeta=1\ \text{on }(-\infty,0],\qquad
\zeta=0\ \text{on }[1,\infty).
\]
For \(R>1\), define
\[
\rho_R(x)
=
\begin{cases}
1, & |x|\leq R^{1/2},\\[1mm]
\displaystyle
\zeta\left(\frac{2\ln |x|}{\ln R}-1\right),
& R^{1/2}<|x|<R,\\[2mm]
0, & |x|\geq R,
\end{cases}
\]
and set
\[
\bar\theta_R(x)=\rho_R^k(x),
\qquad
\bar\phi_R(x)=H_\varepsilon(x)\bar\theta_R(x).
\]

For \(R>\max\{1,R_0^2\}\), we have
\begin{equation}\label{bar-theta-one}
\bar\theta_R=1
\qquad\text{on }\supp\psi.
\end{equation}
Hence,
\begin{equation}\label{bar-phiR-decomp}
\bar\phi_R
=
\bar\theta_R+\varepsilon\psi,
\end{equation}
and consequently
\begin{equation}\label{Delta-bar-phiR}
\Delta\bar\phi_R
=
\Delta\bar\theta_R+\varepsilon\Delta\psi.
\end{equation}
Moreover,
\[
\supp(\Delta\bar\theta_R)
\subset
\overline{B(0,R)\setminus B(0,R^{1/2})}.
\]
Thus, by \eqref{supp-psi-ball},
\begin{equation}\label{bar-disjoint-supports}
\supp(\Delta\bar\theta_R)\cap\supp\psi=\varnothing,
\qquad
\supp(\Delta\bar\theta_R)\cap\supp(\Delta\psi)=\varnothing.
\end{equation}

By \eqref{Heps-bounds} and the support properties of \(\bar\theta_R\),
\begin{equation}\label{bar-phiR-int}
\int_{\mathbb{R}^N}\bar\phi_R(x)\,dx
\preceq R^N.
\end{equation}
Moreover, for \(R\) sufficiently large and
\(R^{1/2}<|x|<R\), the definition of \(\bar\theta_R\) and the
properties of \(\zeta\) yield
\[
|\Delta\bar\theta_R(x)|
\preceq
\frac{1}{|x|^2\ln R}
\zeta^{k-2}\left(\frac{2\ln|x|}{\ln R}-1\right).
\]
Consequently, since \(k>2p'\) and \(0\leq\zeta\leq1\),
\[
\frac{|\Delta\bar\theta_R|^{p'}}
{\bar\theta_R^{p'-1}}
\preceq
\frac{1}{|x|^{2p'}(\ln R)^{p'}}.
\]
Using \(p'=N/2\), we obtain
\begin{equation}\label{bar-thetaR-est}
\int_{\mathbb{R}^N}
\frac{|\Delta\bar\theta_R|^{p'}}
{\bar\theta_R^{p'-1}}\,dx
\preceq
(\ln R)^{-p'}
\int_{R^{1/2}}^R \frac{dr}{r}
\preceq
(\ln R)^{1-p'}.
\end{equation}

From \eqref{bar-phiR-decomp}, \eqref{bar-theta-one}, and
\eqref{bar-disjoint-supports}, we have
\[
\bar\phi_R=\bar\theta_R
\quad\text{on }\supp(\Delta\bar\theta_R),
\qquad
\bar\phi_R=H_\varepsilon
\quad\text{on }\supp(\Delta\psi).
\]
Moreover, by \eqref{Delta-bar-phiR} and
\eqref{bar-disjoint-supports}, the functions
\(\Delta\bar\theta_R\) and \(\varepsilon\Delta\psi\) have disjoint
supports. Therefore,
\[
\int_{\mathbb{R}^N}
\frac{|\Delta\bar\phi_R|^{p'}}
{\bar\phi_R^{p'-1}}\,dx
=
\int_{\supp(\Delta\bar\theta_R)}
\frac{|\Delta\bar\theta_R|^{p'}}
{\bar\theta_R^{p'-1}}\,dx
+
\int_{\supp(\Delta\psi)}
\frac{|\Delta H_\varepsilon|^{p'}}
{H_\varepsilon^{p'-1}}\,dx.
\]
Combining \eqref{Heps-error} and \eqref{bar-thetaR-est}, we obtain
\begin{equation}\label{bar-phiR-est}
\int_{\mathbb{R}^N}
\frac{|\Delta\bar\phi_R|^{p'}}
{\bar\phi_R^{p'-1}}\,dx
\preceq
(\ln R)^{1-p'}+\varepsilon^{p'}.
\end{equation}
Arguing as in the subcritical case, we obtain
\begin{equation}\label{w-bar-phiR-lower}
\int_{\mathbb{R}^N}w(x)\bar\phi_R(x)\,dx
\geq
\frac{\varepsilon A}{2}
\end{equation}
for all sufficiently large \(R\).

For \(R,T\gg1\), define
\[
\bar\varphi_{R,T}(t,x)
=
\eta_T(t)\bar\phi_R(x).
\]
Arguing as in the subcritical case and using
\eqref{etaT-der}, \eqref{bar-phiR-int}, \eqref{etaT-int}, and
\eqref{bar-phiR-est}, we obtain
\[
J_1(\bar\varphi_{R,T})
\preceq
T^{1-p'}R^N,
\qquad
J_2(\bar\varphi_{R,T})
\preceq
T\left((\ln R)^{1-p'}+\varepsilon^{p'}\right).
\]
Moreover, \eqref{etaT-int} and \eqref{w-bar-phiR-lower} give
\[
\varepsilon A T
\preceq
\int_0^\infty\int_{\mathbb{R}^N}
w(x)\bar\varphi_{R,T}(t,x)\,dx\,dt.
\]
Applying \eqref{apriori} with
\(\varphi=\bar\varphi_{R,T}\) and using the above estimates, we obtain
\[
\varepsilon A
\preceq
T^{-p'}R^N
+
(\ln R)^{1-p'}
+
\varepsilon^{p'}.
\]
Choosing \(\varepsilon>0\) sufficiently small allows the last term to
be absorbed into the left-hand side. Hence,
\[
\varepsilon A
\preceq
T^{-p'}R^N+(\ln R)^{1-p'}.
\]
Taking \(T=R^2\ln R\) and recalling that \(p'=N/2\), we obtain
\[
\varepsilon A
\preceq
(\ln R)^{-p'}+(\ln R)^{1-p'}.
\]
Since \(p'>1\), letting \(R\to\infty\) yields a contradiction with
\(\varepsilon A>0\). This completes the proof of
Theorem~\ref{T-main}.
\end{proof}

\section*{Declaration of competing interest}
The authors declare that there is no conflict of interest.

\section*{Data Availability Statements}
The manuscript has no associated data.

\end{document}